\documentclass[12pt]{amsart}
\usepackage{amssymb,amsmath,amsthm,mathrsfs,multirow,xcolor,framed,url}
\usepackage[pdftex,
         pdfauthor={Rong Chen},
         pdftitle={Two-colored generalized Frobenius partitions and minimal-excludant sums over bipartitions},
         pdfsubject={MATHEMATICS},
         pdfkeywords={generalized Frobenius partitions, bipartitions, minimal excludant, Durfee rectangles, staircase partitions},
         pdfproducer={Latex with hyperref},
         pdfcreator={pdflatex}]{hyperref}
\usepackage [latin1]{inputenc}
\newtheorem{theorem}{Theorem}[section]
\newtheorem{lemma}[theorem]{Lemma}
\newtheorem{cor}[theorem]{Corollary}

\theoremstyle{definition}
\newtheorem{definition}[theorem]{Definition}

\theoremstyle{remark}

\numberwithin{equation}{section}

\newcommand\nutwid{\overset {\text{\lower 3pt\hbox{$\sim$}}}\nu}

\allowdisplaybreaks

\newcommand\omycite[1]{}

\newcommand{\beqs}{\begin{equation*}}
\newcommand{\eeqs}{\end{equation*}}
\newcommand{\beq}{\begin{equation}}
\newcommand{\eeq}{\end{equation}}

\newcommand{\mex}{\operatorname{mex}}
\newcommand{\smex}{\sigma\operatorname{mex}}
\newcommand{\cpsi}{c\psi}
\newcommand{\CPsi}{C\Psi}

\title[Generalized Frobenius partitions and bipartition mex sums]{Two-colored generalized Frobenius partitions and minimal-excludant sums over bipartitions}

    \author{Rong Chen}
    \address{Department of Mathematics, Shanghai Normal University, Shanghai, People's Republic of China}
	\email{ rchen@shnu.edu.cn }

	\author{Kang-Yu Wang}
	\address{Department of Mathematics, Shanghai Normal University, Shanghai, People's Republic of China}
	\email{ 1000570883@smail.shnu.edu.cn }

	\author{Yi-Ning Wang}
	\address{Department of Mathematics, Shanghai Normal University, Shanghai, People's Republic of China}
	\email{ 1000570850@smail.shnu.edu.cn }

\begin{document}

\begin{abstract}
Let $\cpsi_{2,a}(n)$ denote the number of $(2,a)$-colored Frobenius partitions of weight $n$, where the two rows have prescribed length difference.  We study the two cases $a=0$ and $a=1$ and connect them with minimal-excludant statistics on bipartitions.  Let $\sigma\mex_2(n)$ be the sum of the Lin--Liu bipartition minimal excludants over all bipartitions of $n$, and let $E_2(n)$ be the number of bipartitions whose two component minimal excludants are equal.  For all $n\geq 0$, we give a combinatorial proof of
\[
\cpsi_{2,0}(n)=2\sigma\mex_2(n)
\qquad\text{and}\qquad
\cpsi_{2,1}(n)=2\sigma\mex_2(n)-E_2(n).
\]
These identities give direct combinatorial interpretations of two-colored Frobenius partition functions in terms of bipartition minimal-excludant sums.
\end{abstract}

\maketitle



\section{Introduction}

We use standard notation for partitions as in \cite{AndrewsPartitions}.  The minimal excludant is a basic statistic on integer partitions.  For a partition $\lambda$, the number $\mex(\lambda)$ is the smallest positive integer that is not a part of $\lambda$.  This statistic records the first gap in the initial sequence of possible parts.  Although the definition is elementary, minimal excludants have become a useful tool in the study of partition identities, congruences, and colored partition functions.

Andrews and Newman \cite{AndrewsNewman} studied the sum of minimal excludants over all partitions of a fixed integer.  Their work connected this sum with two-colored partitions and motivated further investigations of mex-type statistics.  Ballantine and Merca \cite{BallantineMercaCombinatorial} gave a combinatorial proof of the minimal-excludant theorem of Andrews and Newman, and in related work \cite{BallantineMerca} studied least $r$-gaps and identities involving triangular and polygonal numbers.  These results show that first-gap statistics are naturally linked with staircase partitions and classical partition functions.

Bipartitions provide a natural extension of this framework.  If $\pi=(\pi_1,\pi_2)$ is a bipartition of $n$, we write $\pi\vdash_2 n$ and denote by $p_2(n)$ the number of such bipartitions.  Lin and Liu \cite{LinLiu} defined the minimal excludant of $\pi$ by
\[
\mex_2(\pi)=\min\{\mex(\pi_1),\mex(\pi_2)\}.
\]
Equivalently, $\mex_2(\pi)$ is the smallest positive integer that does not occur simultaneously in both components.  The corresponding sum over all bipartitions of $n$ is a direct analogue of the Andrews--Newman minimal-excludant sum, but it also reflects the interaction between the two components of a bipartition.

The second class of objects considered here consists of generalized Frobenius partitions.  Andrews \cite{AndrewsGFP} introduced $k$-colored generalized Frobenius partitions by allowing the entries in Frobenius coordinates to carry colors.  Jiang, Rolen and Woodbury \cite{JiangRolenWoodbury} considered colored two-rowed arrays in which the row lengths may differ by a prescribed amount.  We denote the resulting counting functions by $\cpsi_{k,a}(n)$, following their notation.  In the two-colored case, the parameters $a=0$ and $a=1$ give the two functions studied in this paper.

Our goal is to connect these two settings.  We prove that the two-colored Frobenius-partition counts $\cpsi_{2,0}(n)$ and $\cpsi_{2,1}(n)$ are determined by the Lin--Liu statistic $\mex_2$, together with a simple correction term recording when the two component minimal excludants are equal.  More precisely, let $E_2(n)$ count the bipartitions $\pi=(\pi_1,\pi_2)$ of $n$ for which
\[
\mex(\pi_1)=\mex(\pi_2).
\]
The main identities are
\[
\cpsi_{2,0}(n)=2\sum_{\pi\vdash_2 n}\mex_2(\pi)
\]
and
\[
\cpsi_{2,1}(n)=2\sum_{\pi\vdash_2 n}\mex_2(\pi)-E_2(n),
\]
valid for all $n\geq 0$.  Thus the cases $a=0$ and $a=1$ in the two-colored Frobenius setting admit direct interpretations in terms of the bipartition minimal excludant and an equal-mex correction term.

The paper is organized as follows.  In Section 2, we recall the definitions of the colored Frobenius partitions and bipartition minimal excludants needed below.  In Section 3, we prove the Frobenius-partition identities and the corresponding minimal-excludant identities, and then combine them in the final corollary.

\section{Definitions from Frobenius partitions and minimal excludants}

In this section we recall the definitions used below, following Andrews \cite{AndrewsGFP}, Jiang--Rolen--Woodbury \cite{JiangRolenWoodbury}, Andrews--Newman \cite{AndrewsNewman}, and Lin--Liu \cite{LinLiu}.

\subsection{Generalized Frobenius partitions with prescribed row-length difference}

Ordinary Frobenius coordinates encode a partition as a two-rowed array
\[
\begin{pmatrix}
 a_1&a_2&\cdots&a_m\\
 b_1&b_2&\cdots&b_m
\end{pmatrix},
\]
where
\[
a_1>a_2>\cdots>a_m\geq 0,
\qquad
b_1>b_2>\cdots>b_m\geq 0,
\]
and the weight is
\[
m+\sum_{i=1}^m a_i+\sum_{i=1}^m b_i.
\]
Andrews \cite{AndrewsGFP} generalized this construction by allowing the entries in the two rows to be selected from colored copies of the non-negative integers.  The usual $k$-colored generalized Frobenius partitions are obtained by taking two rows of the same length and selecting entries from $k$ colored copies of the non-negative integers; we denote the number of such objects of weight $n$ by $c\phi_k(n)$ and write
\[
C\Phi_k(q):=\sum_{n\geq 0}c\phi_k(n)q^n.
\]

We now recall the definition of Jiang, Rolen and Woodbury \cite[Definition 1.5]{JiangRolenWoodbury} needed in this paper.  A colored integer is regarded as a pair $(m,j)$, where $m$ is a non-negative integer and $j$ is a color.  The lexicographic order is taken with respect to the usual order on $m$ and the fixed color order $0<1<\cdots<k-1$ on $j$.

\begin{definition}\label{def:kaF}
Let $k\in\mathbb Z_{>0}$ and let $a\in \mathbb Z+k/2$.  A $(k,a)$-colored Frobenius partition is a two-rowed array
\[
\Pi=
\begin{pmatrix}
 a_1&a_2&\cdots&a_r\\
 b_1&b_2&\cdots&b_s
\end{pmatrix}
\]
satisfying the following conditions:
\begin{enumerate}
\item each entry belongs to one of $k$ colored copies of the non-negative integers;
\item each row is strictly decreasing with respect to the lexicographic ordering on colored integers;
\item the row lengths satisfy
\[
r-s=a-\frac{k}{2}.
\]
\end{enumerate}
The weight of $\Pi$ is
\[
|\Pi|=r+\sum_{i=1}^r a_i+\sum_{j=1}^s b_j,
\]
where the colors are ignored in taking the numerical values of the entries.  We denote by $\cpsi_{k,a}(n)$ the number of such arrays of weight $n$ and write
\[
\CPsi_{k,a}(q):=\sum_{n\geq 0}\cpsi_{k,a}(n)q^n.
\]
\end{definition}

For $a=k/2$, the length condition gives $r=s$, and Definition \ref{def:kaF} recovers the usual $k$-colored generalized Frobenius partitions counted by $c\phi_k(n)$.  In this paper we mainly need the cases $(k,a)=(2,0)$ and $(k,a)=(2,1)$.  Thus a $(2,0)$-colored Frobenius partition is a two-rowed array with two colors, say $0$ and $1$, whose row lengths satisfy
\[
r-s=-1.
\]
Similarly, a $(2,1)$-colored Frobenius partition has two colors and row lengths satisfying
\[
r-s=0.
\]

\subsection{Minimal excludants for partitions and bipartitions}

By a partition of $n$ we mean a finite weakly decreasing sequence of positive integers whose sum is $n$; we write $\lambda\vdash n$ and denote its weight by $|\lambda|$.  Let $p(n)$ denote the number of partitions of $n$.  We use the convention that $p(m)=0$ for $m<0$.  If $\lambda$ and $\mu$ are partitions, then $\lambda\cup\mu$ denotes the partition obtained by taking the multiset union of their parts and arranging the parts in weakly decreasing order.

Following Andrews and Newman \cite{AndrewsNewman}, the minimal excludant of a partition $\lambda$ is
\[
\mex(\lambda):=\min\{m\geq 1:m\text{ is not a part of }\lambda\}.
\]
They defined
\[
\smex(n):=\sum_{\lambda\vdash n}\mex(\lambda).
\]

A bipartition of $n$ is an ordered pair $\pi=(\pi_1,\pi_2)$ of partitions with $|\pi_1|+|\pi_2|=n$; we write $\pi\vdash_2 n$.  Let $p_2(n)$ denote the number of bipartitions of $n$, and we use the convention that $p_2(m)=0$ for $m<0$.  Lin and Liu \cite{LinLiu} defined the minimal excludant of a bipartition by
\[
\mex_2(\pi):=\min\{\mex(\pi_1),\mex(\pi_2)\}.
\]
They also defined the corresponding sum
\[
\sigma\mex_2(n):=\sum_{\pi\vdash_2 n}\mex_2(\pi).
\]

\begin{definition}\label{def.equal.mex2}
For $n\geq 0$, let $E_2(n)$ denote the number of bipartitions $\pi=(\pi_1,\pi_2)$ of $n$ whose two component minimal excludants are equal; that is,
\[
E_2(n):=
\#\{\pi=(\pi_1,\pi_2)\vdash_2 n:\mex(\pi_1)=\mex(\pi_2)\}.
\]
\end{definition}

\section{Proofs of the main identities}

\subsection{The Frobenius-partition side}

We first record the one-color case of Definition \ref{def:kaF}.

\begin{lemma}\label{lemma.durfee}
Let $y\in\mathbb Z$ and $N\geq 0$.  Then
\[
\cpsi_{1,\frac12+y}(N)=p\left(N-\frac{y(y+1)}2\right),
\]
where $p(m)=0$ for $m<0$.
\end{lemma}

\begin{proof}
Let
\[
F=
\begin{pmatrix}
 a_1&a_2&\cdots&a_r\\
 b_1&b_2&\cdots&b_s
\end{pmatrix}
\]
be a $(1,\frac12+y)$-colored Frobenius partition of weight $N$.  Since there is only one color, this is a one-color Frobenius symbol with $r-s=y$, where
\[
a_1>a_2>\cdots>a_r\geq 0,
\qquad
b_1>b_2>\cdots>b_s\geq 0,
\]
and
\[
N=r+\sum_{i=1}^{r}a_i+\sum_{j=1}^{s}b_j.
\]
Remove the staircases from the two rows by defining
\[
\alpha_i=a_i-(r-i)\qquad(1\leq i\leq r),
\]
and
\[
\beta_j=b_j-(s-j)\qquad(1\leq j\leq s).
\]
Then
\[
\alpha_1\geq\alpha_2\geq\cdots\geq\alpha_r\geq 0,
\qquad
\beta_1\geq\beta_2\geq\cdots\geq\beta_s\geq 0.
\]

Construct a Ferrers diagram by drawing an $r\times s$ Durfee rectangle, attaching $\alpha_i$ boxes to the right of row $i$ for $1\leq i\leq r$, and attaching $\beta_j$ boxes below column $j$ for $1\leq j\leq s$.  This gives an ordinary partition $\lambda$ with
\[
|\lambda|=rs+\sum_{i=1}^{r}\alpha_i+\sum_{j=1}^{s}\beta_j.
\]
On the other hand,
\[
\sum_{i=1}^{r}a_i=
\sum_{i=1}^{r}\alpha_i+\frac{r(r-1)}2,
\]
and
\[
\sum_{j=1}^{s}b_j=
\sum_{j=1}^{s}\beta_j+\frac{s(s-1)}2.
\]
Therefore
\begin{align*}
N
&=r+\sum_{i=1}^{r}a_i+\sum_{j=1}^{s}b_j\\
&=|\lambda|+\frac{(r-s)(r-s+1)}2\\
&=|\lambda|+\frac{y(y+1)}2.
\end{align*}
Thus $F$ gives a partition of $N-y(y+1)/2$.

The construction is reversible by the standard Durfee-rectangle decomposition: take the Durfee rectangle with height $r$ and width $s$ satisfying $r-s=y$.  The boxes to the right of the rectangle define the partition $(\alpha_i)$, and the boxes below the rectangle define the partition $(\beta_j)$.  Adding back the staircases,
\[
a_i=\alpha_i+(r-i),
\qquad
b_j=\beta_j+(s-j),
\]
recovers a unique one-color Frobenius symbol of weight $N$ and length difference $y$.
\end{proof}

We now express $\cpsi_{2,0}(n)$ and $\cpsi_{2,1}(n)$ in terms of the bipartition function.

\begin{theorem}\label{thm.cpsi.p2}
For all $n\geq 0$, we have
\[
\cpsi_{2,0}(n)=\sum_{d\in\mathbb Z}p_2\bigl(n-d(d+1)\bigr)
\]
and
\[
\cpsi_{2,1}(n)=\sum_{d\in\mathbb Z}p_2(n-d^2).
\]
\end{theorem}

\begin{proof}
Let $a\in\{0,1\}$, and let $\Pi$ be a $(2,a)$-colored Frobenius partition of weight $n$.  For $i\in\{0,1\}$, let $r_i$ be the number of entries of color $i$ in the top row, and let $s_i$ be the number of entries of color $i$ in the bottom row.  Since $\Pi$ is a $(2,a)$-colored Frobenius partition, its total row lengths satisfy
\[
(r_0+r_1)-(s_0+s_1)=a-1.
\]
Put
\[
d_i=r_i-s_i\qquad(i=0,1).
\]
Then
\[
d_0+d_1=a-1.
\]

For each color $i$, we extract the subarray consisting only of entries of color $i$, and then forget the color.  This gives a one-color Frobenius symbol of length difference $d_i$.  By Lemma \ref{lemma.durfee}, this subarray corresponds to an ordinary partition $\lambda_i$, with weight decreased by
\[
\frac{d_i(d_i+1)}2.
\]
Therefore $\Pi$ corresponds to a pair of ordinary partitions $(\lambda_0,\lambda_1)$ whose total weight is
\[
|\lambda_0|+|\lambda_1|
=
 n-\frac{d_0(d_0+1)}2-\frac{d_1(d_1+1)}2.
\]
Writing $d=d_0$ and $d_1=a-1-d$, this total weight is
\[
n-\frac{d(d+1)}2-\frac{(a-1-d)(a-d)}2.
\]

Conversely, given an integer $d$ and a bipartition $(\lambda_0,\lambda_1)$ of
\[
n-\frac{d(d+1)}2-\frac{(a-1-d)(a-d)}2,
\]
apply the inverse of the Durfee rectangle construction in Lemma \ref{lemma.durfee} to $\lambda_0$ with length difference $d$ and to $\lambda_1$ with length difference $a-1-d$.  Coloring the resulting arrays by $0$ and $1$, and then merging the entries in each row according to the lexicographic order on colored integers, recovers a unique $(2,a)$-colored Frobenius partition of weight $n$.

If $a=0$, then
\[
\frac{d(d+1)}2+\frac{(-1-d)(-d)}2=d(d+1),
\]
and hence
\[
\cpsi_{2,0}(n)=\sum_{d\in\mathbb Z}p_2\bigl(n-d(d+1)\bigr).
\]
If $a=1$, then
\[
\frac{d(d+1)}2+\frac{(-d)(1-d)}2=d^2,
\]
and hence
\[
\cpsi_{2,1}(n)=\sum_{d\in\mathbb Z}p_2(n-d^2).
\]
\end{proof}

\subsection{The minimal excludant side}

Next we give the corresponding minimal-excludant interpretations of the two $p_2$-sums in Theorem \ref{thm.cpsi.p2}.  The first is the bipartition analogue of the staircase argument for least gaps, directly parallel to the proof of Ballantine and Merca \cite[Theorem 1.1]{BallantineMerca} in the case $r=1$.

\begin{theorem}\label{thm.sigma.p2}
For all $n\geq 0$, we have
\[
2\sigma\mex_2(n)=\sum_{d\in\mathbb Z}p_2\bigl(n-d(d+1)\bigr)
\]
and
\[
2\sigma\mex_2(n)-E_2(n)=\sum_{d\in\mathbb Z}p_2(n-d^2).
\]
\end{theorem}

\begin{proof}
Let
\[
\delta(k)=(k,k-1,\ldots,2,1)
\]
be the staircase partition with largest part $k$, and put $\delta(0)=\emptyset$.  Then
\[
|\delta(k)|=1+2+\cdots+k=\frac{k(k+1)}2.
\]
For $N\in\mathbb Z$, let $\mathcal B(N)$ denote the set of bipartitions of $N$, with the convention that $\mathcal B(N)=\emptyset$ if $N<0$.  Thus $|\mathcal B(N)|=p_2(N)$.

We first prove the identity involving $\sigma\mex_2(n)$.  For $d\in\mathbb Z$, define
\[
\ell(d)=
\begin{cases}
d,&d\geq 0,\\
-d-1,&d<0.
\end{cases}
\]
Then
\[
d(d+1)=\ell(d)(\ell(d)+1)=2|\delta(\ell(d))|.
\]
For fixed $n$ and $d$, define
\[
\varphi_{n,d}:\mathcal B\bigl(n-d(d+1)\bigr)\longrightarrow \mathcal B(n)
\]
by
\[
\varphi_{n,d}(\mu_1,\mu_2)=
(\mu_1\cup\delta(\ell(d)),\mu_2\cup\delta(\ell(d))).
\]
This map adds one copy of each of the parts $1,2,\ldots,\ell(d)$ to both components.  Hence the image consists exactly of those bipartitions $\pi=(\pi_1,\pi_2)$ of $n$ for which each of the parts $1,2,\ldots,\ell(d)$ occurs in both $\pi_1$ and $\pi_2$.  Equivalently,
\[
\operatorname{Im}(\varphi_{n,d})
=A_{n,d}:=
\{\pi=(\pi_1,\pi_2)\vdash_2 n: \mex_2(\pi)>\ell(d)\}.
\]
The inverse map on $A_{n,d}$ is obtained by removing one copy of each of the parts $1,2,\ldots,\ell(d)$ from both components.  Thus $\varphi_{n,d}$ is a bijection from $\mathcal B(n-d(d+1))$ onto $A_{n,d}$, and therefore
\[
|A_{n,d}|=p_2\bigl(n-d(d+1)\bigr).
\]
Consequently,
\[
\sum_{d\in\mathbb Z}p_2\bigl(n-d(d+1)\bigr)
=
\sum_{d\in\mathbb Z}|A_{n,d}|.
\]
The right-hand side counts pairs $(d,\pi)$ such that $d\in\mathbb Z$, $\pi\vdash_2 n$, and $\pi\in A_{n,d}$.  We may count the same pairs by fixing $\pi$ first.  If $\mex_2(\pi)=m$, then $\pi\in A_{n,d}$ precisely when $\ell(d)<m$.  There are exactly $2m$ integers $d$ satisfying this condition, namely
\[
d=0,1,\ldots,m-1
\qquad\text{and}\qquad
 d=-1,-2,\ldots,-m.
\]
Thus each bipartition $\pi$ contributes $2\mex_2(\pi)$ to the sum.  Hence
\[
\sum_{d\in\mathbb Z}p_2\bigl(n-d(d+1)\bigr)
=
\sum_{\pi\vdash_2 n}2\mex_2(\pi)
=
2\sigma\mex_2(n).
\]

We now prove the identity involving $E_2(n)$.  The term $d=0$ in
\[
\sum_{d\in\mathbb Z}p_2(n-d^2)
\]
counts all bipartitions of $n$.  For $k\geq 1$, the term $d=k$ is represented by the map
\[
\psi^+_{n,k}:\mathcal B(n-k^2)\longrightarrow \mathcal B(n),
\]
where
\[
\psi^+_{n,k}(\mu_1,\mu_2)
=
(\mu_1\cup\delta(k),\mu_2\cup\delta(k-1)),
\]
and the term $d=-k$ is represented by the map
\[
\psi^-_{n,k}:\mathcal B(n-k^2)\longrightarrow \mathcal B(n),
\]
where
\[
\psi^-_{n,k}(\mu_1,\mu_2)
=
(\mu_1\cup\delta(k-1),\mu_2\cup\delta(k)).
\]
Both maps increase the weight by $k^2$, since
\[
|\delta(k)|+|\delta(k-1)|
=
\frac{k(k+1)}2+\frac{k(k-1)}2
=k^2.
\]
The image of $\psi^+_{n,k}$ is
\[
B^+_{n,k}
=
\{\pi=(\pi_1,\pi_2)\vdash_2 n:
\mex(\pi_1)>k,\ \mex(\pi_2)\geq k\},
\]
because one removes one copy of each of $1,2,\ldots,k$ from the first component and one copy of each of $1,2,\ldots,k-1$ from the second component to recover the preimage.  Similarly, the image of $\psi^-_{n,k}$ is
\[
B^-_{n,k}
=
\{\pi=(\pi_1,\pi_2)\vdash_2 n:
\mex(\pi_1)\geq k,\ \mex(\pi_2)>k\}.
\]
Thus $\psi^+_{n,k}$ and $\psi^-_{n,k}$ are bijections from $\mathcal B(n-k^2)$ onto $B^+_{n,k}$ and $B^-_{n,k}$, respectively.  It follows that
\[
\sum_{d\in\mathbb Z}p_2(n-d^2)
=
|\mathcal B(n)|+
\sum_{k\geq 1}|B^+_{n,k}|+
\sum_{k\geq 1}|B^-_{n,k}|.
\]
This identity is the point at which the bijections convert the $p_2$-sum into a weighted count of bipartitions of $n$: the left-hand side counts preimages in the various sets $\mathcal B(n-d^2)$, while the right-hand side counts their images in $\mathcal B(n)$, with multiplicity according to how many of the above image sets contain a fixed bipartition.

Fix a bipartition $\pi=(\pi_1,\pi_2)$ of $n$, and put
\[
m_1=\mex(\pi_1),
\qquad
m_2=\mex(\pi_2).
\]
The bipartition $\pi$ is counted once by the term $d=0$.  For $k\geq 1$, it is counted by $B^+_{n,k}$ precisely when
\[
m_1>k,
\qquad
m_2\geq k,
\]
and it is counted by $B^-_{n,k}$ precisely when
\[
m_1\geq k,
\qquad
m_2>k.
\]
Therefore the total multiplicity of $\pi$ in the above sum is
\[
1+\#\{k\geq 1:m_1>k,\ m_2\geq k\}
 +\#\{k\geq 1:m_1\geq k,\ m_2>k\}.
\]
The two cardinalities are
\[
\#\{k\geq 1:m_1>k,\ m_2\geq k\}=\min{m_1-1,m_2}
\]
and
\[
\#\{k\geq 1:m_1\geq k,\ m_2>k\}=\min{m_1,m_2-1}.
\]
Thus the total multiplicity is
\[
1+\min{m_1-1,m_2}+\min{m_1,m_2-1}.
\]
If $m_1=m_2=m$, this is $2m-1$.  If $m_1\neq m_2$, this is $2\min\{m_1,m_2\}$.  Since
\[
\mex_2(\pi)=\min\{m_1,m_2\},
\]
the multiplicity can be written uniformly as
\[
2\mex_2(\pi)-\mathbf{1}_{\{\mex(\pi_1)=\mex(\pi_2)\}},
\]
where $\mathbf{1}_E$ is $1$ if the condition $E$ holds and is $0$ otherwise.  Summing this multiplicity over all bipartitions $\pi$ of $n$ gives
\[
\sum_{d\in\mathbb Z}p_2(n-d^2)
=
\sum_{\pi\vdash_2 n}
\left(2\mex_2(\pi)-\mathbf{1}_{\{\mex(\pi_1)=\mex(\pi_2)\}}\right).
\]
By the definitions of $\sigma\mex_2(n)$ and $E_2(n)$, the last expression is
\[
2\sigma\mex_2(n)-E_2(n).
\]
This completes the proof.
\end{proof}

\begin{cor}\label{cor.main.identities}
For all $n\geq 0$, we have
\[
\cpsi_{2,0}(n)=2\sigma\mex_2(n)
\]
and
\[
\cpsi_{2,1}(n)=2\sigma\mex_2(n)-E_2(n).
\]
\end{cor}

\begin{proof}
The first identity follows by comparing the first identities in Theorems \ref{thm.cpsi.p2} and \ref{thm.sigma.p2}.  The second identity follows by comparing the second identities in Theorems \ref{thm.cpsi.p2} and \ref{thm.sigma.p2}.
\end{proof}


\begin{thebibliography}{99}

\bibitem{AndrewsPartitions}
G. E. Andrews, \textit{The Theory of Partitions}, Cambridge Mathematical Library, Cambridge University Press, Cambridge, 1998.

\bibitem{AndrewsGFP}
G. E. Andrews, \textit{Generalized Frobenius Partitions}, Memoirs of the American Mathematical Society, Vol. 49, No. 301, American Mathematical Society, Providence, RI, 1984.

\bibitem{AndrewsNewman}
G. E. Andrews and D. Newman, Partitions and the minimal excludant, \textit{Annals of Combinatorics} \textbf{23} (2019), 249--254.

\bibitem{BallantineMercaCombinatorial}
C. Ballantine and M. Merca, Combinatorial proof of the minimal excludant theorem, \textit{arXiv:1908.06789} (2020).

\bibitem{BallantineMerca}
C. Ballantine and M. Merca, Bisected theta series, least $r$-gaps in partitions, and polygonal numbers, \textit{Ramanujan Journal} \textbf{52} (2020), 433--444.

\bibitem{JiangRolenWoodbury}
Y. Jiang, L. Rolen and M. Woodbury, Generalized Frobenius partitions, Motzkin paths, and Jacobi forms, \textit{Journal of Combinatorial Theory, Series A} \textbf{190} (2022), Article 105633.

\bibitem{LinLiu}
B. L. S. Lin and H. Liu, The minimal excludant of bipartitions, \textit{Revista de la Real Academia de Ciencias Exactas, F\'{i}sicas y Naturales. Serie A. Matem\'{a}ticas} \textbf{120} (2026), Article 18.

\end{thebibliography}
\end{document}